\documentclass[11pt]{article}
\usepackage{amsmath,amsfonts,amssymb,amsthm}
\usepackage{enumerate}
\usepackage{caption}
\usepackage{pbox}
\usepackage[pagebackref,colorlinks,citecolor=blue,linkcolor=blue]{hyperref}

\numberwithin{equation}{section}
\theoremstyle{plain}
\newtheorem{theorem}[equation]{Theorem}

\newtheorem{lemma}[equation]{Lemma}
\newtheorem{proposition}[equation]{Proposition}

\usepackage{geometry}
\theoremstyle{definition}
\newtheorem{definition}[equation]{Definition}

\newtheorem{example}[equation]{Example}
\newtheorem{remark}[equation]{Remark}

\numberwithin{equation}{section}

\newcommand{\R}{{\mathbb R}}
\newcommand{\N}{{\mathbb N}}

\newcommand{\Om}{\Omega}

\providecommand{\vint}[1]{\mathchoice
          {\mathop{\vrule width 5pt height 3 pt depth -2.5pt
                  \kern -9pt \kern 1pt\intop}\nolimits_{\kern -5pt{#1}}}
          {\mathop{\vrule width 5pt height 3 pt depth -2.6pt
                  \kern -6pt \intop}\nolimits_{\kern -3pt{#1}}}
          {\mathop{\vrule width 5pt height 3 pt depth -2.6pt
                  \kern -6pt \intop}\nolimits_{\kern -3pt{#1}}}
          {\mathop{\vrule width 5pt height 3 pt depth -2.6pt
                  \kern -6pt \intop}\nolimits_{\kern -3pt{#1}}}}

\newcommand{\eps}{\varepsilon}
\newcommand{\loc}{\mathrm{loc}}

\newcommand{\BV}{\mathrm{BV}}

\newcommand{\liploc}{\mathrm{Lip}_{\mathrm{loc}}}

\newcommand{\ch}{\text{\raise 1.3pt \hbox{$\chi$}\kern-0.2pt}}

\DeclareMathOperator{\Mod}{Mod}
\DeclareMathOperator{\capa}{Cap}

\DeclareMathOperator{\dist}{dist}
\DeclareMathOperator{\diam}{diam}

\DeclareMathOperator{\Lip}{Lip}

\DeclareMathOperator{\inte}{int}

\begin{document}
\title{On rough traces of BV functions
\footnote{{\bf 2010 Mathematics Subject Classification}: 30L99, 26B30, 46E35
\hfill \break {\it Keywords\,}: boundary trace, rough trace, function of bounded variation,
metric measure space,  integrability, extension
}}
\author{Panu Lahti}
\maketitle

\begin{abstract}
In metric measure spaces, we study boundary traces of BV functions in domains equipped with a doubling measure
and supporting a Poincar\'e inequality, but possibly having a very large and irregular boundary.
We show that the trace exists in the ordinary sense in a certain part of the boundary,
and that this part is sufficient to determine the integrability of the rough trace,
as well as the possibility of zero extending the function to the whole space as
a BV function.
\end{abstract}

\section{Introduction}

Boundary traces of Sobolev and BV (bounded variation) functions are a relevant concept for example in the study of Dirichlet boundary
value problems. Classical treatments of boundary traces,
or traces for short, of BV functions in Euclidean spaces can be found e.g. in
\cite[Chapter 3]{AFP} and \cite[Chapter 5]{EvGa}.
Traditionally, one considers domains with a Lipschitz boundary.
On the other hand, traces can be studied also in more general domains, and
in abstract metric measure spaces $(X,d,\mu)$.
Usually one assumes that the measure $\mu$ satisfies a doubling property and
that the space supports a Poincar\'e inequality.
In such a setting it is natural to define the trace as follows: given a function
$u$ defined on an open set $\Om\subset X$, the number
$Tu(x)\in\R$ is the boundary trace of $u$ at $x\in\partial\Om$ if
\[
\lim_{r\to 0}\,\vint{\Om\cap B(x,r)}|u-Tu(x)|\,d\mu=0.
\]
In \cite{LaSh,MSS}, various properties of traces of BV functions in metric spaces were shown.
Despite being far more general than the classical setting of Lipschitz domains,
the theory in these papers still required rather strong regularity of $\Om$ and especially of its boundary.
If one assumes very little or no regularity of $\Om$, the ordinary boundary trace might not exist
on the entire boundary,
but the following \emph{rough trace} always exists:
\begin{align*}
T_*u(x):=\sup\left\{t\in\R\colon \theta^*(\{u>t\},x)>0\right\},\quad x\in\partial\Om.
\end{align*}
Here we denote by $\theta^*$ the upper measure density; see Section \ref{sec:definitions}
for definitions.
Rough traces have been studied by Maz'ya \cite[Section 9.5]{Maz} in Euclidean spaces,
and by Buffa--Miranda \cite{BuMi} in metric spaces.

In the current paper, we are interested in considering open sets $\Om\subset X$
that have some regularity but significantly less than those considered \cite{LaSh,MSS}.
This means in particular that the boundary $\partial\Om$ can be very large and that
the (ordinary) trace $Tu$ might only exist on a
part of it. Nonetheless, this part turns out to be enough to largely determine the behavior
of the rough trace $T_*u$ on the entire boundary.
We will assume that, in a suitable sense, $\Om$ is equipped with a doubling measure
and supports a Poincar\'e inequality locally near its boundary.
We abbreviate this by saying that $\Om$ is PLB, see Definition \ref{def:local conditions on boundary}.

Our main result is the following.
We define $\Om_{\beta}$ to be the set where the lower measure density of $\Om$ is at least a constant
$\beta>0$, and $\Sigma_{\beta}\Om:=\Om_{\beta}\cap (X\setminus \Om)_{\beta}$.

\begin{theorem}\label{thm:main theorem}
Suppose $\Om\subset X$ is PLB and let $u\in\BV(\Om)$. Then
\begin{enumerate}[(1)]
\item the trace $Tu(x)$ exists at
$\mathcal H$-a.e. $x\in\partial\Om\cap \Om_{\beta}$;
\item the integrals of the boundary traces satisfy
\[
\int_{\partial\Om\cap \{\theta^*(\Om,\cdot)>0\}}|T_{*}u|\,d\mathcal H=\int_{\partial\Om\cap \Om_{\beta}}|Tu|\,d\mathcal H;\ \ \textrm{and}
\]
\item the zero extension of $u$ from $\Om$ to $X$ is in $\BV(X)$ if and only if
\[
\int_{\Sigma_{\beta}\Om}|Tu|\,d\mathcal H<\infty.
\]
\end{enumerate}
\end{theorem}

In essence, the theorem says that the trace exists in the ordinary sense on a part of the boundary,
namely $\partial\Om\cap\Om_{\beta}$, and that
this is enough to determine the integrability of the rough trace,
as well as the zero extension property of the function.
We will prove the three different parts of the theorem
in Sections \ref{sec:boundary traces}, \ref{sec:integrability}, and \ref{sec:zero extension}.
In Section \ref{sec:zero extension} we also examine the possibility of zero extending $u$ as a $\BV$ function without
adding any total variation.

\section{Preliminaries}\label{sec:definitions}

In this section we introduce the basic notation, definitions,
and assumptions that are employed in the paper.

Throughout this paper, $(X,d,\mu)$ is a complete metric space that is equip\-ped
with a metric $d$ and a Borel regular outer measure $\mu$ satisfying
a doubling property, meaning that
there exists a constant $C_d\ge 1$ such that
\[
0<\mu(B(x,2r))\le C_d\mu(B(x,r))<\infty
\]
for every open ball $B(x,r):=\{y\in X\colon d(y,x)<r\}$, with $x\in X$ and $r>0$.
We assume that $X$ consists of at least $2$ points.
When a property holds outside of a set of $\mu$-measure zero, we say that it holds at
$\mu$-a.e. point.

All functions defined on $X$ or its subsets will take values in $[-\infty,\infty]$.
As a complete metric space equipped with a doubling measure, $X$ is proper,
that is, closed and bounded sets are compact.
A function $u$ defined in an open set $\Omega\subset X$
is said to be in the class $L^1_{\loc}(\Omega)$ if it is in $L^1(\Om')$ for
every open $\Omega'\Subset\Omega$.
Here $\Omega'\Subset\Omega$ means that $\overline{\Omega'}$ is a
compact subset of $\Omega$.
Other local spaces of functions are defined analogously.

By a curve we mean a rectifiable continuous mapping from a compact interval of the real line into $X$.
The length of a curve $\gamma$
is denoted by $\ell_{\gamma}$. We will assume every curve to be parametrized
by arc-length (see e.g. \cite[Theorem~3.2]{Hj}), so we have $\gamma\colon [0,\ell_{\gamma}]\to X$.
A nonnegative Borel function $g$ on $X$ is an upper gradient 
of a function $u$
on $X$ if for all nonconstant curves $\gamma$, we have
\begin{equation}\label{eq:definition of upper gradient}
	|u(x)-u(y)|\le \int_{\gamma} g\,ds:=\int_0^{\ell_{\gamma}} g(\gamma(s))\,ds,
\end{equation}
where $x$ and $y$ are the end points of $\gamma$.
We interpret $|u(x)-u(y)|=\infty$ whenever  
at least one of $|u(x)|$, $|u(y)|$ is infinite.
Upper gradients were originally introduced in \cite{HK}.

We say that a domain $\Omega\subset X$ is $M$-uniform, with constant $M\ge 1$,
if for every $x,y\in\Omega$ there exists a curve $\gamma\colon [0,\ell_{\gamma}]\to \Omega$ with 
$\ell_{\gamma}\le Md(x,y)$, $\gamma(0)=x$, $\gamma(\ell_{\gamma})=y$, and
such that for all $t\in [0,\ell_{\gamma}]$ we have
\begin{equation}\label{eq:uniformity}
	\dist(\gamma(t),X\setminus\Omega)\ge M^{-1}\min\{t,\ell_{\gamma}-t\}.
\end{equation}
We say that a domain is uniform if it is $M$-uniform for some $1\le M<\infty$.

Let $1\le p <\infty$. (We will work almost exclusively with $p=1$.)
The $p$-modulus of a family of curves $\Gamma$ is defined by
\[
\Mod_{p}(\Gamma):=\inf\int_{X}\rho^p\, d\mu,
\]
where the infimum is taken over all nonnegative Borel functions $\rho$
such that $\int_{\gamma}\rho\,ds\ge 1$ for every curve $\gamma\in\Gamma$.
A property is said to hold for $p$-almost every curve
if it fails only for a curve family with zero $p$-modulus. 
If $g$ is a nonnegative $\mu$-measurable function on $X$
and (\ref{eq:definition of upper gradient}) holds for $p$-almost every curve,
then we say that $g$ is a $p$-weak upper gradient of $u$. 
By only considering curves $\gamma$ in a set $A\subset X$,
we can talk about a function $g$ being a ($p$-weak) upper gradient of $u$ in $A$.

Given an open set $\Om\subset X$, we let
\[
\Vert u\Vert_{N^{1,p}(\Om)}:=\left(\int_\Om |u|^p\,d\mu+\inf \int_\Om g^p\,d\mu \right)^{1/p},
\]
where the infimum is taken over all $p$-weak upper gradients $g$ of $u$ in $\Om$.
Then we define the Newton-Sobolev space
\[
N^{1,p}(\Om):=\{u\colon \|u\|_{N^{1,p}(\Om)}<\infty\},
\]
which was first introduced in \cite{S}.
For any $u\in N^{1,p}(\Om)$ the quantity $\Vert u\Vert_{N^{1,p}(\Om)}$
agrees with the classical Sobolev norm, see e.g. \cite[Corollary A.4]{BB}.
It is known that for every $u\in N_{\loc}^{1,p}(\Om)$ there exists a minimal $p$-weak
upper gradient of $u$ in $\Om$, which we always denote by $g_{u}$, satisfying $g_{u}\le g$ 
$\mu$-almost everywhere in $\Om$ for every $p$-weak upper gradient $g\in L_{\loc}^{p}(\Om)$
of $u$ in $\Om$, see \cite[Theorem 2.25]{BB}.

For any open sets $\Om,\Om_0\subset X$, the space of Newton-Sobolev functions
with zero boundary values is defined as
\begin{equation}\label{eq:definition of N011}
	N_0^{1,p}(\Om,\Om_0):=\{u|_{\Om\cap \Om_0}\colon \,u\in N^{1,p}(\Om_0)\textrm{ and }u=0\textrm { in }\Om_0\setminus \Om\}.
\end{equation}
This space is a subspace of $N^{1,p}(\Om\cap \Om_0)$, and it can
be understood to be a subspace of $N^{1,p}(\Om_0)$ as well.
We let $N_0^{1,p}(\Om):=N_0^{1,p}(\Om,X)$.

The $p$-capacity of a set $A\subset X$ is defined by
\begin{equation}\label{eq:definition of capacity}
\capa_p(A):=\inf \Vert u\Vert_{N^{1,p}(X)}^p,
\end{equation}
where the infimum is taken over all functions $u\in N^{1,p}(X)$ such that $u\ge 1$ in $A$.

For any set $A\subset X$ and $0<R<\infty$, the Hausdorff content
of codimension one is defined by
\[
\mathcal{H}_{R}(A):=\inf\left\{ \sum_{j=1}^{\infty}
\frac{\mu(B(x_{j},r_{j}))}{r_{j}}:\,A\subset\bigcup_{j=1}^{\infty}B(x_{j},r_{j}),\,r_{j}\le R\right\}.
\]
We also allow finite coverings by interpreting $\mu(B(x,0))/0=0$.
The codimension one Hausdorff measure of $A\subset X$ is then defined by
\[
\mathcal{H}(A):=\lim_{R\rightarrow 0}\mathcal{H}_{R}(A).
\]
We also define ``centered'' versions $\widehat{\mathcal H}_R$ and $\widehat{\mathcal H}$
in the same way, but with the additional requirement that $x_j\in A$. By the doubling property of $\mu$,
these are comparable to $\mathcal{H}_{R}$ and $\mathcal{H}$.

By \cite[Theorem 4.3, Theorem 5.1]{HaKi} we know that for every $A\subset X$,
\begin{equation}\label{eq:null sets of Hausdorff measure and capacity}
	\capa_1(A)=0\quad \textrm{if and only if}\quad \mathcal H(A)=0.
\end{equation}

We will assume throughout the paper that $X$ supports a $(1,1)$-Poincar\'e inequality,
meaning that there exist constants $C_P>0$ and $\lambda \ge 1$ such that for every
ball $B(x,r)$, every $u\in L^1_{\loc}(X)$,
and every upper gradient $g$ of $u$,
we have
\[
	\vint{B(x,r)}|u-u_{B(x,r)}|\, d\mu 
	\le C_P r\vint{B(x,\lambda r)}g\,d\mu,
\]
where 
\[
u_{B(x,r)}:=\vint{B(x,r)}u\,d\mu :=\frac 1{\mu(B(x,r))}\int_{B(x,r)}u\,d\mu.
\]

Next we present the definition and basic properties of functions
of bounded variation on metric spaces, following Miranda Jr.
\cite{M}.
Given an open set $\Om\subset X$ and a function $u\in L^1_{\loc}(\Om)$,
we define the total variation of $u$ in $\Om$ by
\begin{equation}\label{eq:total variation}
\|Du\|(\Om):=\inf\left\{\liminf_{i\to\infty}\int_\Om g_{u_i}\,d\mu\colon \, u_i\in N^{1,1}_{\loc}(\Om),\, u_i\to u\textrm{ in } L^1_{\loc}(\Om)\right\},
\end{equation}
where each $g_{u_i}$ is the minimal $1$-weak upper gradient of $u_i$ in $\Om$.
We say that a function $u\in L^1(\Om)$ is of bounded variation, 
and denote $u\in\BV(\Om)$, if $\|Du\|(\Om)<\infty$.
For an arbitrary set $A\subset X$, we define
\[
\|Du\|(A):=\inf\{\|Du\|(W)\colon A\subset W,\,W\subset X
\text{ is open}\}.
\]
In \cite{M}, local Lipschitz constants were used in place of upper gradients, but the theory
can be developed similarly with either definition. It
is sometimes also required that $u_i\in \liploc(\Om)$ instead of $u_i\in N^{1,1}_{\loc}(\Om)$,
but for us this does not make a difference, since $\liploc(\Om)$ is dense in $N^{1,1}_{\loc}(\Om)$,
see \cite[Theorem 4.57]{BB}.

If $u\in L^1_{\loc}(\Om)$ and $\Vert Du\Vert(\Omega)<\infty$,
then $\|Du\|(\cdot)$ is
a Borel regular outer measure on $\Omega$ by \cite[Theorem 3.4]{M}.
A $\mu$-measurable set $E\subset X$ is said to be of finite perimeter if $\|D\ch_E\|(X)<\infty$, where $\ch_E$ is the characteristic function of $E$.
The perimeter of $E$ in $\Omega$ is also denoted by
\[
P(E,\Omega):=\|D\ch_E\|(\Omega).
\]

We define the lower and upper densities of a set $E\subset X$ at a point $x\in X$ as follows:
\[
\theta_*(E,x):=\liminf_{r\to 0}\frac{\mu(B(x,r)\cap E)}{\mu(B(x,r))}
\quad\textrm{and}\quad 
\theta^*(E,x):=\limsup_{r\to 0}\frac{\mu(B(x,r)\cap E)}{\mu(B(x,r))}.
\]
The measure-theoretic interior of a set $E\subset X$ is defined by
\[
	I_E:=
	\left\{x\in X\colon\theta^*(X\setminus E,x)=0\right\},
\]
and the measure-theoretic exterior by
\[
	O_E:=
	\left\{x\in X\colon\theta^*(E,x)=0\right\}.
\]
The measure-theoretic boundary is defined as the set of points
$x\in X$
at which both $E$ and its complement have strictly positive upper density:
\[
\partial^{*}E:=\left\{x\in X\colon \theta^*(E,x)>0
\textrm{ and }
\theta^*(X\setminus E,x)>0\right\}.
\]
Note that the space $X$ is always partitioned into the disjoint sets
$I_E$, $O_E$, and $\partial^*E$.
We also let
\begin{equation}\label{eq:definition of Omega beta}
E_{b}:=\{x\in X\colon \theta_{*}(E,x)\ge b\},\quad b> 0.
\end{equation}
The strong boundary $\Sigma_{b}E$, for $0<b\le 1/2$,
is defined as $\Sigma_{b}E:=E_{b}\cap (X\setminus E)_{b}$.

For an open set $\Omega\subset X$ and a $\mu$-measurable set $E\subset X$ with $P(E,\Omega)<\infty$, we know that
\begin{equation}\label{eq:measure theoretic and strong boundary}
	\mathcal H((\partial^*E\setminus \Sigma_{\gamma}E)\cap \Om)=0
\end{equation}
for some number $\gamma=\gamma(C_d,C_P,\lambda)>0$, see \cite[Theorem 5.4]{A1}.
We also know that for any Borel set $A\subset\Omega$,
\begin{equation}\label{eq:def of theta}
	P(E,A)=\int_{\Sigma_{\gamma}E\cap A}\theta_E\,d\mathcal H,
\end{equation}
where
$\theta_E\colon \Om\to [\alpha,C_d]$ with $\alpha=\alpha(C_d,C_P,\lambda)>0$, see \cite[Theorem 5.3]{A1} 
and \cite[Theorem 4.6]{AMP}.
In particular, $P(E,\Om)<\infty$ implies that $\mathcal H(\partial^*E\cap\Om)<\infty$.
Federer's characterization of sets of finite perimeter states that the converse is also true;
more precisely, if $E\subset X$ is a $\mu$-measurable set such that $\mathcal H(\partial^*E\cap \Om)<\infty$,
then $P(E,\Om)<\infty$,
see \cite[Theorem 1.1]{L-Fedchar}.
See also Federer \cite[Section 4.5.11]{Fed} for the original Euclidean result.

The strong boundary can also be used to characterize sets of finite perimeter, as follows.

\begin{theorem}[{\cite[Theorem 1.1]{L-newFed}}]\label{thm:new Federer}
	Let $\Om\subset X$ be an open set and let $E\subset X$ be a $\mu$-measurable set
	with $\mathcal H(\Sigma_{\beta} E\cap \Om)<\infty$,
	where $0<\beta\le 1/2$ only depends on the doubling constant of the measure
	and the constants in the Poincar\'e inequality. Then $P(E,\Om)<\infty$.
\end{theorem}

Throughout this paper, we will use $\beta$ to denote the constant from this theorem;
we can assume that $\beta\le \gamma$.
Now combining Theorem \ref{thm:new Federer} and \eqref{eq:def of theta}, we obtain that for every
open $\Om\subset X$ and $\mu$-measurable $E\subset X$, we have
\[
	P(E,\Om)\le C\mathcal H(\Sigma_{\beta}E\cap\Om)
\]
for some constant $C\ge 1$ depending only on the constants $C_d,C_P,\lambda$.
By combining \eqref{eq:measure theoretic and strong boundary} and \eqref{eq:def of theta},
we also get
\[
\mathcal H(\Sigma_{\beta}E\cap\Om)\le CP(E,\Om).
\]
In total, we have
\begin{equation}\label{eq:perimeter and strong boundary}
\frac{1}{C}\mathcal H(\Sigma_{\beta}E\cap\Om)\le P(E,\Om)\le C\mathcal H(\Sigma_{\beta}E\cap\Om)
\end{equation}
for some constant $C\ge 1$ depending only on the constants $C_d,C_P,\lambda$.

For a function $u$ defined on an open set $\Om\subset X$,
we will often abbreviate super-level sets in the form
\[
\{u>t\}:=\{x\in \Om\colon u(x)>t\},\quad t\in\R.
\]
The following coarea formula is given in \cite[Proposition 4.2]{M}:
if $\Omega\subset X$ is open and $u\in L^1_{\loc}(\Omega)$, then
\begin{equation}\label{eq:coarea}
	\|Du\|(\Omega)=\int_{-\infty}^{\infty}P(\{u>t\},\Omega)\,dt.
\end{equation}
The integral should be understood as an upper integral; however if either side is finite, then
both sides are finite and the integrand is measurable.

The lower and upper approximate limits of a function $u$ on an open set
$\Om$
are defined respectively by
\[
	u^{\wedge}(x):
	=\sup\left\{t\in\R:\,\lim_{r\to 0}\frac{\mu(B(x,r)\cap\{u<t\})}{\mu(B(x,r))}=0\right\}
\]
and
\begin{equation}\label{eq:upper approximate limit}
	u^{\vee}(x):
	=\inf\left\{t\in\R:\,\lim_{r\to 0}\frac{\mu(B(x,r)\cap\{u>t\})}{\mu(B(x,r))}=0\right\}
\end{equation}
for $x\in \Om$.
We use the usual convention that the supremum and infimum of an empty set are
$-\infty$ and $\infty$, respectively.
The jump set of $u$ is then defined by
\[
S_u:=\{u^{\wedge}<u^{\vee}\}.
\]
Since we understand $u^{\wedge}$ and $u^{\vee}$ to be defined only on $\Om$,
also $S_u$ is understood to be a subset of $\Om$.
It is straightforward to check that $u^{\wedge}$ and $u^{\vee}$
are always Borel functions.

The following Lebesgue point result was proved in
\cite[Theorem 3.5]{KKST}: if $\Om\subset X$ is open and $u\in\BV(\Om)$, then
for $\mathcal H$-a.e. $x\in\Om$ we have
\begin{equation}\label{eq:Lebesgue points}
\lim_{r\to 0}\,\vint{B(x,r)}|u-u^{\vee}(x)|\,d\mu=0.
\end{equation}
(Equally well we could replace $u^{\vee}(x)$ by $u^{\wedge}(x)$.)

By \cite[Theorem 5.3]{AMP}, the variation measure of a $\BV$ function
can be decomposed into the absolutely continuous and singular part, and the latter
into the Cantor and jump part, as follows. Given an open set 
$\Omega\subset X$ and $u\in\BV(\Omega)$, we have for any Borel set $A\subset \Om$
\begin{equation}\label{eq:variation measure decomposition}
	\begin{split}
		\Vert Du\Vert(A) &=\Vert Du\Vert^a(A)+\Vert Du\Vert^s(A)\\
		&=\Vert Du\Vert^a(A)+\Vert Du\Vert^c(A)+\Vert Du\Vert^j(A)\\
		&=\int_{A}a\,d\mu+\Vert Du\Vert^c(A)+\int_{A\cap S_u}\int_{u^{\wedge}(x)}^{u^{\vee}(x)}\theta_{\{u>t\}}(x)\,dt\,d\mathcal H(x),
	\end{split}
\end{equation}
where $a\in L^1(\Omega)$ is the density of the absolutely continuous part
and the functions $\theta_{\{u>t\}}\in [\alpha,C_d]$ 
are as in \eqref{eq:def of theta}.\\

\emph{Throughout this paper we assume that $(X,d,\mu)$ is a complete metric space
	that is equipped with the doubling measure $\mu$ and supports a
	$(1,1)$-Poincar\'e inequality.}

\section{Existence of boundary traces}\label{sec:boundary traces}

In this section we study existence results for boundary traces on a certain part of the boundary.
The symbol $\Om\subset X$ always denotes an arbitrary open set.
Recall that when $u$ is a function defined on $\Om$, we denote
\[
\{u>t\}:=\{x\in \Om\colon u(x)>t\}\subset \Om,\quad t\in\R.
\]

We start with the definitions of the boundary traces, or traces for short, that we will study throughout the paper.

\begin{definition}\label{def:traces}
	Let $u$ be a $\mu$-measurable function on $\Om$.
	
The rough trace of $u$ at $x\in\partial\Om$ is
\begin{align*}
T_*u(x):=\sup\{t\in\R\colon \theta^*(\{u>t\},x)>0\}.
\end{align*}
As before, we interpret the supremum of an empty set to be $-\infty$.

We define a second version of the rough trace at $x\in\partial\Om$ by
\begin{align*}
T_{\beta}u(x):=\sup\{t\in\R\colon \theta_*(\{u>t\},x)\ge \beta\},
\end{align*}
where $0<\beta\le 1/2$ is the constant from Theorem \ref{thm:new Federer}.

Finally, if for $x\in \partial\Om$ there exists $b\in\R$ such that
\[
\lim_{r\to 0}\,\vint{\Om\cap B(x,r)}|u-b|\,d\mu=0,
\]
then we say that $Tu(x):=b$ is an (ordinary) trace of $u$ at $x$.
\end{definition}

\begin{remark}
In the classical definition of the rough trace given in \cite[Section 9.5]{Maz}, in the supremum 
it is required that $x\in \partial^*\{u>t\}$,
but we replace this with the weaker requirement $\theta^*(\{u>t\},x)>0$,
since we wish to consider also points $x\in\partial\Om$ where $\Om$ has density one.
Note that if $u\in\BV(\Om)$, it follows from the coarea formula \eqref{eq:coarea} that
\[
P(\{u>t\},\Om)<\infty\quad \textrm{for a.e. }t\in\R.
\]
In the classical definition, it is additionally required
that $\{u>t\}$ has finite perimeter in the whole space, that is $P(\{u>t\},X)<\infty$.
We do not wish to require this, since we consider very general open sets $\Om$
that could have infinite perimeter in
$X$, and then the super-level sets $\{u>t\}$ often also have infinite perimeter in $X$.
However, our definition coincides with the classical one $\mathcal H$-almost everywhere on the boundary
under the assumptions that $P(\Om,X)<\infty$ and $\mathcal H(\partial\Om\setminus \partial^*\Om)=0$,
which are assumed in the classical theory.

It is straightforward to check that all three traces are Borel functions on
$\partial\Om$ (in the case of $Tu$, more precisely it is Borel
on the subset of $\partial\Om$ where it is defined), and so
integrals with respect to the Borel outer measure $\mathcal H$ are well-defined.
In terms of comparisons between the traces, clearly $T_{\beta}u(x) \le T_*u(x)$
for all $x\in\partial\Om$. If $x\in \partial\Om$ such that $\theta_*(\Om,x)>0$ and $Tu(x)$ exists,
then it is easy to check that $T_*u(x)=T_{\beta}u(x)=Tu(x)$.
\end{remark}

We start by giving a simple example demonstrating that in
completely general open sets $\Om$, the (ordinary) trace $Tu$ might not exist at any point of the boundary $\partial\Om$.
For this reason, in our main results we will assume that $\Om$ has some regularity, formulated in terms
of local doubling and Poincar\'e conditions.

\begin{example}\label{ex:bad traces}
Let $C\subset [0,1]$ be the ternary Cantor set. Let $\Om:=(0,1)\setminus C$. Then
$\Om=\bigcup_{j=1}^{\infty}\Om_j$, where each $\Om_j$ consists of $2^{j-1}$ open intervals of length $3^{-j}$.
Let
\[
u:=\sum_{j=1}^{\infty}b_j\ch_{\Om_j},\quad b_j:=
\begin{cases}
j & \textrm{for }j\textrm{ odd,}\\
-j & \textrm{for }j\textrm{ even.}
\end{cases}
\]
Note that
\[
\Vert u\Vert_{L^1(\Om)}
= \frac{1}{2}\sum_{j=1}^{\infty}j \left(\frac{2}{3}\right)^j
<\infty.
\]
Moreover $\Vert Du\Vert(\Om)=0$, since $u$ is locally constant. Thus $u\in \BV(\Om)$.

Note that $\partial\Om=C$. Now for every $x\in C$ and every $t>0$,
for every odd $j> t$ we have that $B(x,3^{-j+1})$ contains an interval of length
$3^{-j}$ and belonging to $\{u>t\}$.
Denoting the $1$-dimensional Lebesgue measure by $\mathcal L^1$, we get
\[
\limsup_{r\to 0}\frac{\mathcal L^1(B(x,r)\cap \{u>t\})}{\mathcal L^1(B(x,r))}\ge \frac{1}{4}
\]
and similarly
\[
\limsup_{r\to 0}\frac{\mathcal L^1(B(x,r)\cap \{u<-t\})}{\mathcal L^1(B(x,r))}\ge \frac{1}{4}.
\]
Thus for every $x\in C$, we have $T_*u(x)=\infty$ while the trace $Tu(x)$ fails to exist.
Moreover, due to the alternating between negative and positive values
in $b_j$, the trace would still fail to exist with any reasonable definition that would allow
$Tu(x)$ to take the values $\pm \infty$.
\end{example}

In the rest of the work, most of the time we will
consider open sets $\Om$ that satisfy certain regularity near the boundary, at least locally.
In order to define such $\Om$, we will consider subsets $A\subset \overline{\Om}$
as metric spaces in their own right
(including the case $A=\overline{\Om}$).

\begin{definition}
	For any $A\subset \overline{\Om}$, we define the metric measure space $(A,d,\mu_A)$
	as follows.
The metric $d$ is simply inherited from $X$.
When $A\subset \Om$, we equip it with
the measure $\mu$ restricted to subsets of $A$.
This restriction is a Borel regular outer measure on 
$A$ by \cite[Lemma~3.3.11]{HKST}.
When $A\subset \overline{\Om}$ is Borel, we equip it with
the zero extension of $\mu$ from $\Omega\cap A$ to $A$, denoted by $\mu_A$.
That is, for every $D\subset A$ we have $\mu_{A}(D)=\mu(D\cap \Om)$.
By \cite[Lemma~3.3.16]{HKST} we know that $\mu_{A}$ is a Borel regular outer measure on $A$.
We denote a ball in the space $A$ by $B_A(x,r):=B(x,r)\cap A$.
We also denote by $\mathcal H_{A,R}$, $R>0$, and $\mathcal H_{A}$ the codimension one Hausdorff content and measure
in the space $(A,d,\mu_A)$.
\end{definition}

\begin{definition}\label{def:local conditions on boundary}
We say that the open set $\Om$ is \emph{PLB} ($1$-Poincar\'e space locally near its boundary)
if for $\mathcal H$-a.e. $x\in \partial\Om\cap \Om_{\beta}$ (recall \eqref{eq:definition of Omega beta})
there exists an open set $W\subset \Om$ such that $(W,d,\mu_W)$
	is a metric space for which $\mu_W$ is doubling and $W$ supports a $(1,1)$-Poincar\'e inequality,
	and $x\in V$ for some relatively open set $V\subset \overline{\Om}$ with $V\cap \Om\subset W$.
\end{definition}

Briefly, we will sometimes say that ``$(W,d,\mu_W)$ satisfies doubling and $(1,1)$-Poincar\'e''.
Given sets $V$ and $W$ as above, note that
\begin{equation}\label{eq:boundaries of Omega and W}
\partial\Om\cap V = \partial W\cap V.
\end{equation}

\begin{remark}\label{rem:locally Poincare near boundary}
	There is a wide range of domains $\Om$ that satisfy doubling and Poincar\'e globally
	(more precisely, $(\Om,d,\mu_\Om)$ satisfies doubling and $(1,1)$-Poincar\'e),
but for us it will be enough to assume the significantly weaker local conditions
as in Definition \ref{def:local conditions on boundary}. In this way, we include more open sets $\Om$
in our theory,
and even in cases where doubling and Poincar\'e hold globally in $\Om$, the local conditions are often much easier
to check.
In particular, for many Euclidean domains $\Om$ it is easy to check that for $\mathcal H$-a.e.
$x\in \partial\Om\cap \Om_{\beta}$, for sufficiently small $r>0$ the set $W:=B(x,r)\cap\Om$
has a Lipschitz boundary and is thus a uniform domain.
(Note that in the Euclidean space $(\R^n,d_{\textrm{euc}},\mathcal L^n)$, the codimension one Hausdorff measure
$\mathcal H$ is comparable to the $n-1$-dimensional Hausdorff measure $\mathcal H^{n-1}$.)
By \cite[Theorem 4.4]{BS}, the space $(W,d,\mu)$ then satisfies doubling and $(1,1)$-Poincar\'e.
We will use these facts in Examples \ref{ex:measure theoretic and strong boundary}
and \ref{ex:strong boundary}.
\end{remark}

\begin{lemma}\label{lem:Hausdorff zero measure sets agree}
	Let $V$ and $W$ be two sets as in Definition \ref{def:local conditions on boundary}.
	Let $A\subset V$.
	Then $\mathcal H_{\overline{W}}(A)=0$ implies that also $\widehat{\mathcal H}_{\overline{\Om}}(A)=0$.
\end{lemma}
Here $\widehat{\mathcal H}_{\overline{\Om}}(A)$ is the ``centered'' version of 
$\mathcal H_{\overline{\Om}}$, that is, in the definition the coverings are required to be centered in the set
$A$.

Note that from Definition \ref{def:local conditions on boundary} we get $V\subset \overline{W}$, so then also
$A\subset \overline{W}$.
\begin{proof}
Suppose $\mathcal H_{\overline{W}}(A)=0$.
Let
\[
V_{\delta}:=\{x\in V\colon d(x,\overline{\Om}\setminus V)>\delta\},\quad \delta>0.
\]
We have $V=\bigcup_{j=1}^{\infty}V_{1/j}$, and so it is enough to prove that
$\widehat{\mathcal H}_{\overline{\Om}}(V_{\delta}\cap A)=0$
for an arbitrary but fixed $\delta>0$.

Fix also $0<\eps<\delta$.
We find a covering $\{B_{\overline{W}}(x_j,r_j)\}_{j=1}^{\infty}$ of $V_{\delta}\cap A$
in the space $\overline{W}$ such that $r_j<\eps$ and
\[
\sum_{j=1}^{\infty}\frac{\mu_{\overline{W}}(B_{\overline{W}}(x_j,r_j))}{r_j}<\eps.
\]
By the doubling property of $\mu_{\overline{W}}$,
we can assume that  $x_j\in V_{\delta}\cap A$.
Thus $B_{\overline{\Om}}(x_j,r_j)\subset V\subset \overline{W}$,
so that $\overline{W}\cap B(x_j,r_j)=\overline{\Om}\cap B(x_j,r_j)$.
Then by \eqref{eq:boundaries of Omega and W}, we have $W\cap B(x_j,r_j)=\Om\cap B(x_j,r_j)$. Thus
\[
\mu_{\overline{W}}(B_{\overline{W}}(x_j,r_j))
=\mu(W \cap B(x_j,r_j))
=\mu(\Om\cap B(x_j,r_j))
=\mu_{\overline{\Om}}(B_{\overline{\Om}}(x_j,r_j)).
\]
Thus also
\[
\widehat{\mathcal H}_{\overline{\Om},\eps}(V_{\delta}\cap A)\le \sum_{j=1}^{\infty}\frac{\mu_{\overline{\Om}}(B_{\overline{\Om}}(x_j,r_j))}{r_j}<\eps.
\]
Since $\eps>0$ was arbitrary, we get $\widehat{\mathcal H}_{\overline{\Om}}(V_{\delta}\cap A)=0$.
\end{proof}

\begin{lemma}\label{lem:H and overlineH}
	Suppose $A\subset \overline{\Om}$ such that $\theta_*(\Om,x)>0$ for all $x\in A$.
	Then $\mathcal H(A)=0$ if and only if $\widehat{\mathcal H}_{\overline{\Om}}(A)=0$.
\end{lemma}
\begin{proof}
	First suppose that $\mathcal H(A)=0$. Fix $\eps>0$.
	There exists a covering $\{B(x_k,r_k)\}_{k=1}^{\infty}$ of $A$ with $r_k<\eps/2$ and
	\begin{equation}\label{eq:covering of a choice}
	\sum_{k=1}^{\infty}\frac{\mu(B(x_k,r_k))}{r_k}<\frac{\eps}{C_d^2}.
	\end{equation}
	We can assume that every ball in the covering intersects $A$.
	Thus for every $k\in\N$ we find a point $y_k\in B(x_k,r_k)\cap A$.
	In particular $y_k\in \overline{\Om}$, and $\{B_{\overline{\Om}}(y_k,2r_k)\}_{k=1}^{\infty}$
	is a covering of $A$ in the space $\overline{\Om}$, with $B(y_k,2r_k)\subset B(x_k,3r_k)$. Thus we get
	\[
	\widehat{\mathcal H}_{\overline{\Om},\eps}(A)
	\le \sum_{k=1}^{\infty}\frac{\mu_{\overline{\Om}}(B_{\overline{\Om}}(y_k,2r_k))}{2r_k}
	\le \sum_{k=1}^{\infty}\frac{\mu(B(x_k,3r_k))}{r_k}<\eps
	\]
	by \eqref{eq:covering of a choice}.
	It follows that $\widehat{\mathcal H}_{\overline{\Om}}(A)=0$.
	
	Conversely, suppose that $\widehat{\mathcal H}_{\overline{\Om}}(A)=0$.
	For each $j\in\N$, let
	\[
	A_j:=\left\{x\in A\colon \inf_{0<r\le 1/j}\frac{\mu(\Om\cap B(x,r))}{\mu(B(x,r))}\ge \frac{1}{j}\right\}.
	\]
	Then $A=\bigcup_{j=1}^{\infty}A_j$,
	and of course $\widehat{\mathcal H}_{\overline{\Om}}(A_j)=0$ for all $j\in\N$.
	Fix $j\in\N$ and fix $0<\eps<1/j$.
	There exists a covering $\{B_{\overline{\Om}}(x_k,r_k)\}_{k=1}^{\infty}$ of $A_j$ with $x_k\in A_j$,
	$r_k< \eps$, and
	\[
	\sum_{k=1}^{\infty}\frac{\mu_{\overline{\Om}}(B_{\overline{\Om}}(x_k,r_k))}{r_k}<\frac{\eps}{j}.
	\]
	Now $\{B(x_k,r_k)\}_{k=1}^{\infty}$ is a covering of $A_j$ in $X$, and so
	\[
	\mathcal H_{\eps}(A_j)\le \sum_{k=1}^{\infty}\frac{\mu(B(x_k,r_k))}{r_k}
	\le j\sum_{k=1}^{\infty}\frac{\mu(\Om\cap B(x_k,r_k))}{r_k}
	=j\sum_{k=1}^{\infty}\frac{\mu_{\overline{\Om}}(B_{\overline{\Om}}(x_k,r_k))}{r_k}
	<\eps.
	\]
	In conclusion, $\mathcal H(A_j)=0$.
	It follows that
	\[
	\mathcal H(A) \le \sum_{j=1}^{\infty}\mathcal H(A_j)=0.
	\]
\end{proof}

The statement of the following lemma is quite close to the definition of the total variation
\eqref{eq:total variation};
the idea is simply that we can have $u_i\in N^{1,1}(\Om)$ and convergence in $L^1(\Om)$,
instead of only $u_i\in N_{\loc}^{1,1}(\Om)$ and convergence in $L_{\loc}^1(\Om)$.

\begin{lemma}\label{lem:N11 in definition of total variation}
	Let $u\in\BV(\Om)$. Then there exists a sequence of functions $u_i\in N^{1,1}(\Om)$ with
	$u_i\to u$ in $L^1(\Om)$ and
	\[
	\Vert Du\Vert(\Om)=\lim_{i\to\infty}\int_{\Om}g_{u_i}\,d\mu,
	\]
	where each $g_{u_i}$ is the minimal $1$-weak upper gradient of $u_i$ in $\Om$.
\end{lemma}
\begin{proof}
	This is given as part of \cite[Corollary 6.7]{LaSh}. Note that the ``strong relative isoperimetric inequality''
	mentioned there is proved in \cite[Corollary 5.6]{L-Fedchar}.
\end{proof}

Now we prove the existence of the trace $Tu$ on a part of the boundary.
This gives Claim (1) of Theorem \ref{thm:main theorem}.

\begin{proposition}\label{prop:existence of trace}
	Suppose $\Om$ is PLB.
	Let $u\in\BV(\Om)$.
	Then the trace $Tu(x)$ exists for $\mathcal H$-a.e. $x\in\partial \Om\cap \Om_{\beta}$.
\end{proposition}

\begin{proof}
	By definition, we can cover $\mathcal H$-almost all of $\partial \Om\cap \Om_{\beta}$ by
	relatively open sets $V\subset\overline{\Om}$ such that $V\subset \overline{W}$
	for open sets $W\subset \Om$ for which each $(W,d,\mu_{W})$
	satisfies doubling and $(1,1)$-Poincar\'e.
	Since $X$ is Lindel\"of and thus so are its subsets, we find a countable collection of
	these sets $\{W_j\}_{j=1}^{\infty}$ such that the corresponding sets $V_j$ cover
	$\mathcal H$-almost all of $\partial \Om\cap \Om_{\beta}$.
	Fix $V_j$ and $W_j$ and denote them $V$ and $W$.
	By Lemma \ref{lem:Hausdorff zero measure sets agree}, among subsets of $V$,
	being a null set with respect to $\mathcal H_{\overline{W}}$ implies
	being a null set with respect to $\widehat{\mathcal H}_{\overline{\Om}}$,
	and by Lemma \ref{lem:H and overlineH}, among subsets of $\Om_{\beta}$,
	the measures $\widehat{\mathcal H}_{\overline{\Om}}$ and $\mathcal H$ have the same null sets.
	Thus it is enough to prove that $Tu(x)$ exists for $\mathcal H_{\overline{W}}$-a.e.
	$x\in \partial \Om\cap V$.
	
	Denote by $\overline{u}$ the zero extension of $u$ from $W$ to $\overline{W}$.
	The function $\overline{u}$ remains measurable in the space
	$\overline{W}=(\overline{W},d,\mu_{\overline{W}})$, see \cite[Lemma 3.3.18]{HKST}.
	Clearly $\Vert \overline{u}\Vert_{L^1(\overline{W})}=\Vert u\Vert_{L^1(W)}$.
	By Lemma \ref{lem:N11 in definition of total variation}, we find a sequence of functions
	$u_i\in N^{1,1}(W)$ such that $u_i\to u$ in $L^1(W)$ and
	\[
	\Vert Du\Vert(W)=\liminf_{i\to \infty}\int_{W}g_{u_i}\,d\mu.
	\]
	Since $(W,d,\mu_W)$ satisfies doubling and $(1,1)$-Poincar\'e,
	we know that Lipschitz functions $\Lip(W)$ are dense in $N^{1,1}(W)$, see \cite[Theorem 5.1]{BB}. 
	Thus we can in fact assume that $u_i\in \Lip(W)$.
	Now we can extend every $u_i$ to a Lipschitz function on 
	$\overline{W}$, still denoted by $u_i$.
	By \cite[Lemma 5.11]{BS},
	the zero extension of $g_{u_i}$ to $\overline{W}$, still denoted by the same symbol,
	is the minimal $1$-weak upper gradient of $u_i$ in $\overline{W}$.
	Now $u_i\to \overline{u}$ in $L^1(\overline{W})$ and so
	\[
	\Vert Du\Vert(W)=\liminf_{i\to \infty}\int_{W}g_{u_i}\,d\mu
	=\liminf_{i\to \infty}\int_{\overline{W}}g_{u_i}\,d\mu_{\overline{W}}\ge \Vert D\overline{u}\Vert(\overline{W}).
	\]
	Thus we have $\overline{u}\in \BV(\overline{W})$ with
	$\Vert D\overline{u}\Vert(\overline{W})=\Vert Du\Vert(W)$, 
	so it follows that $\Vert D\overline{u}\Vert(\partial W)=0$.
	By \eqref{eq:boundaries of Omega and W} we have that $\partial W\cap V=\partial \Om\cap V$.
	Thus in total, we have
	\begin{equation}\label{eq:u BV in W closure}
	\overline{u}\in \BV(\overline{W})\quad \textrm{with}\quad \Vert D\overline{u}\Vert(\partial\Om\cap V)=0.
	\end{equation}
	By \cite[Proposition 3.3 \& 3.6]{BB-noncomp} we know that $(\overline{W},d,\mu_{\overline{W}})$
	also satisfies doubling and $(1,1)$-Poincar\'e.
	Since $\Vert D\overline{u}\Vert(\partial \Om\cap V)=0$,
	by the decomposition \eqref{eq:variation measure decomposition} we also
	know that $\mathcal H_{\overline{W}}(S_{\overline{u}}\cap\partial \Om\cap V)=0$.
	By the Lebesgue point result \eqref{eq:Lebesgue points}, for
	$\mathcal H_{\overline{W}}$-a.e. $x\in\overline{W}\setminus S_{\overline{u}}$ we have
	\[
	\lim_{r\to 0}\,\vint{\Om\cap B(x,r)}|u-\overline{u}^\vee(x)|\, d\mu
	= \lim_{r\to 0}\,\vint{B_{\overline{W}}(x,r)}|\overline{u}-\overline{u}^\vee(x)|\, d\mu_{\overline{W}}=0.
	\]
	In particular, this now holds for $\mathcal H_{\overline{W}}$-a.e. $x\in\partial \Om\cap V$.
	Thus we can choose $Tu(x):=\overline{u}^\vee(x)$.
\end{proof}

In the proof above, we used Lemma \ref{lem:H and overlineH} to compare the measures
$\widehat{\mathcal H}_{\overline{\Om}}$ and $\mathcal H$; of course it is
more natural to formulate results in terms of the latter.
In \cite{LaSh} it was noted that these two measures
have the same null sets on $\partial\Om$ if
$\Omega$ satisfies a \emph{measure density condition}, meaning that there is a constant $c_m>0$ such that
\[
	\mu(B(x,r)\cap\Omega)\ge c_m\mu(B(x,r))
\]
for all  $x\in\partial\Omega$ and $r\in (0,\diam(\Omega))$.
We wish to avoid such an assumption that of course rules out e.g. domains with exterior cusps.
Instead, we have considered the part of the boundary where
the measure density condition is satisfied in an asymptotic sense, namely $\partial\Om\cap \Om_{\beta}$.
Concerning the rest of the boundary, we first observe that the trace might not exist.

\begin{example}\label{ex:measure theoretic and strong boundary}
	Let $X=\R^2$.
	For $A\subset \R^2$ and $x\in\R^2$, we define
	\[
	A+x:=\{y+x\colon y\in A\}.
	\]
	Define $\Om\subset \R^2$ as a union of rectangles ``piled on top of each other'', as follows.
	Define
	\[
	A_j:=[0,2^{-j(j+1)}]\times [0,2^{-j^2}]+c_j
	\]
	with
	\[
	c_j:=\left(-2^{-j(j+1)-1},\,\sum_{k=1}^{j-1}2^{-k^2}\right),\quad j\in\N.
	\]
	Then let $A:=\bigcup_{j=1}^{\infty}A_j$ and $\Om:=\inte(A)$.
	Consider the point
	\[
	x_0:=\left(0,\sum_{k=1}^{\infty}2^{-k^2}\right)\in \partial\Om.
	\]
	Equip the space with the weighted Lebesgue measure
	\[
	d\mu:=w\,d\mathcal L^2,\quad\textrm{with}\quad w(x)=\frac{1}{2\pi}|x-x_0|^{-1}.
	\]
	It is straightforward to check that $w$ is a Muckenhoupt $A_1$-weight,
	and thus $\mu$ is doubling and supports a $(1,1)$-Poincar\'e inequality,
	see e.g. \cite[Chapter 15]{HKM} for these concepts.  
	We will show that $x_0\in\partial^*\Om$ but $\theta_*(\Om,x_0)=0$.
	First note that
	\[
	\mu(B(x_0,r))=\frac{1}{2\pi}\int_0^r \frac{2\pi t}{t}\,dt=r, \quad r>0.
	\]
	Denote by $Q_j$ the square of side length $2^{-j(j+1)}$ located at the top of the rectangle $A_j$.
	Note that the ball $B(x_0,2^{-j(j+1)+1})$ contains $Q_j$. Thus for every $j\in\N$,
	\begin{equation}\label{eq:suitable scales one}
	\begin{split}
	\frac{\mu(B(x_0,2^{-j(j+1)+1})\cap \Om)}{\mu(B(x_0,2^{-j(j+1)+1}))}
	&\ge \frac{\mu(Q_j)}{2^{-j(j+1)+1}}\\
	&\ge \frac{1}{2\pi}\frac{\mathcal L^2(Q_j)/2^{-j(j+1)+1}}{2^{-j(j+1)+1}}\\
	&\ge \frac{1}{2\pi}\frac{2^{-j(j+1)}\times 2^{-j(j+1)}/2^{-j(j+1)+1}}{2^{-j(j+1)+1}}
	=\frac{1}{8\pi}.
	\end{split}
	\end{equation}
	This combined with the obvious fact that
	$\theta^*(\R^2\setminus \Om,x_0)\ge \theta_*(\R^2\setminus \Om,x_0)\ge 1/2$ tells us that
	$x_0\in\partial^*\Om$.
	On the other hand, we let $S_j:=\sum_{k=j+1}^{\infty}2^{-k^2}$ and then we estimate
	\begin{equation}\label{eq:measure of Aj}
	\begin{split}
	\mu(A_j)
	&\le 2^{-j(j+1)}\times \int_{S_j}^{S_j+2^{-j^2}}t^{-1}\,dt\\
	&\le 2^{-j(j+1)}\times \log\left((S_j+2^{-j^2})/S_j\right)\\
		&\le 2^{-j(j+1)}\times \log\left(1+2^{-j^2}/2^{-(j+1)^2}\right)\\
	&\le 2^{-j(j+1)}\times \log\left(1+2^{3j}\right).
	\end{split}
	\end{equation}
	Thus for every $j\in\N$, we have
	\begin{align*}
	\frac{\mu(B(x_0,2^{-j^2})\cap \Om)}{\mu(B(x_0,2^{-j^2}))}
	&\le \frac{1}{2^{-j^2}}\sum_{k=j}^{\infty}\mu(A_k)\\
	&\le \frac{1}{2^{-j^2}}\sum_{k=j}^{\infty}2^{-k(k+1)}\times \log\left(1+2^{3k}\right)\\
	&\le \frac{2}{2^{-j^2}}2^{-j(j+1)}\times \log\left(1+2^{3j}\right)\\
	&\to 0\quad\textrm{as }j\to\infty.
	\end{align*}
	This tells us that $\theta_*(\Om,x_0)=0$, so in particular
	$x_0\in\partial^*\Om\setminus \Sigma_{\beta}\Om$.
	
	Next note that for every $x\in \partial\Om\setminus \{x_0\}$,
	we can choose $\delta>0$ so small that $B(x,\delta)$ only intersects at most two of the rectangles that make
	up $\Om$. Then $W:=B(x,\delta)\cap \Om$ is obviously
	a Lipschitz domain.
	As noted in Remark \ref{rem:locally Poincare near boundary},
	$(W,d,\mu)$ is then a metric space equipped with a doubling measure and supporting a
	$(1,1)$-Poincar\'e inequality. 
	Thus $\Om$ is PLB.
	
	Define a function
	\[
	u:=\sum_{k=1}^{\infty}(-1)^k k\ch_{A_{k}}.
	\]
	Then by \eqref{eq:measure of Aj}, we get
	\[
	\Vert u\Vert_{L^1(\Om)}
	\le \sum_{k=1}^{\infty}k\mu(\ch_{A_{k}})
	\le \sum_{k=1}^{\infty}k\cdot 2^{-k(k+1)}\times \log\left(1+2^{3k}\right)<\infty,
	\]
	and
	\[
	\Vert Du\Vert (\Om)\le \sum_{k=1}^{\infty}2k\cdot 2^{-k(k+1)}<\infty.
	\]
	Thus $u\in\BV(\Om)$. Now, for any even numbers $j\in\N$ and $k\ge j$, we have
	\begin{align*}
	\frac{\mu(B(x_0,2^{-k(k+1)+1})\cap \{u\ge j\})}{\mu(B(x_0,2^{-k(k+1)+1}))}
	\ge \frac{\mu(Q_k)}{2^{-k(k+1)+1}}\ge \frac{1}{8\pi}\quad\textrm{by }\eqref{eq:suitable scales one}.
	\end{align*}
	Thus $\theta^{*}(\{u\ge j\},x_0)\ge (8\pi)^{-1}$ for all even $j\in\N$, and so $T_*u(x_0)=\infty$.
	On the other hand, we have $\theta_*(\{u>t\},x_0)=0$
	for all $t\in\R$ since we proved that in fact $\theta_*(\Om,x_0)=0$, and so
	$T_{\beta}u(x_0)=-\infty$. Finally, clearly $Tu(x_0)$ does not exist,
	nor would it exist with any reasonable definition allowing for the possibility that $Tu(x_0)$
	take the values $\pm \infty$.
	To estimate $\mathcal H(\{x_0\})$ it is of course enough to consider coverings consisting of one ball,
	and then it is straightforward to show that
	\begin{equation}\label{eq:measure of origin}
	\frac{1}{4}\le \mathcal H(\{x_0\})\le 1.
	\end{equation}
\end{example}

This example demonstrates that there can be a significant part of the boundary,
in terms of $\mathcal H$-measure, where $\theta^*(\Om,\cdot)>0$ but $\theta_*(\Om,\cdot)=0$.
On this part, the trace $Tu$ might not exist but the rough trace $T_*u$ is of course defined.
We will see in the next sections that the values of $T_*u$ on this part of the boundary are
nonetheless controlled by the part of the boundary where $Tu$ exists.
(If there is a part of $\partial\Om$ where even $\theta^*(\Om,\cdot)=0$, there we have
$T_*u\equiv -\infty$ and so the rough trace is not interesting.)

\section{Integrability of the boundary trace}\label{sec:integrability}

In this section we study the integrability of the rough trace $T_*u$.
As usual, $\Om\subset X$ always denotes an arbitrary open set.

In a bounded Lipschitz domain $\Om\subset\R^n$, by Lusin's theorem the boundary trace $Tu$
has the following continuity: for every $\eps>0$ there exists a relatively open set $G\subset \partial\Om$
such that $\mathcal H^{n-1}(G)<\eps$ and $Tu|_{\partial\Om\setminus G}$ is continuous.
In more general domains in metric spaces,
it is difficult to obtain a similar result, already because
$\partial\Om$ could have infinite $\mathcal H$-measure.
Nonetheless, we get the following kind of continuity.

\begin{proposition}\label{prop:weak continuity}
	Let $u\in\BV(\Om)$.
	Then for $\mathcal H$-a.e. $x\in\partial \Om\setminus \Om_{\beta}$, we have that if
	$T_*u(x)>-\infty$ and $\eps>0$, then
	\[
	\mathcal H(B(x,\eps )\cap \{T_{\beta}u\ge \min\{T_{*}u(x)-\eps,1/\eps\}\})=\infty.
	\]
\end{proposition}

Note that we take the minimum because it can happen that $T_*u(x)=\infty$,
in which case we of course interpret $T_*u(x)-\eps=\infty$.
The above condition means that close to $x$ there is a very large subset of $\partial\Om$
where $T_{\beta}u$ is not much less than $T_{*}u(x)$ (though conversely, recall that always $T_{\beta}u\le T_*u$).

\begin{proof}
	Since $X$ and then also $\partial\Om$ is Lindel\"of, we find
	a countable collection of balls $\{B_j\}_{j=1}^{\infty}$ such that every $x\in\partial\Om$ is contained
	in some ball $B_j$ with arbitrarily small radius. 
	By the coarea formula \eqref{eq:coarea}, we can take a countable,
	dense subset $\{q_k\}_{k\in\N}$ of $\R$ such that
	\[
	P(\{u>q_k\},\Om)<\infty\quad\textrm{for all}\ k\in\N.
	\]
	For each $j,k\in\N$, if $P(\{u>q_k\},B_j)<\infty$, then define
	\[
	N_{j,k}:=B_j\cap (\partial^*\{u>q_k\}\setminus \Sigma_{\beta}\{u>q_k\}).
	\]
	Otherwise let $N_{j,k}:=\emptyset$.
	Then define the exceptional set
	\[
	N:=\bigcup_{j,k=1}^{\infty}N_{j,k}.
	\]
	We have $\mathcal H(N)=0$ by \eqref{eq:measure theoretic and strong boundary};
	recall that we assume $\beta\le \gamma$.
	
	Now fix $x\in \partial\Om\setminus (\Om_{\beta}\cup N)$.
	Suppose $T_*u(x)\in (-\infty,\infty]$ and let $\eps>0$.
	Denote $b:=\min\{T_{*}u(x)-\eps,1/\eps\}\in (-\infty,\infty)$.
	Suppose that for some fixed $\eps>0$,
	\begin{equation}\label{eq:finite measure contradiction}
		\mathcal H(B(x,\eps )\cap \{T_{\beta}u\ge b\})<\infty.
	\end{equation}
	Note that for every point $y\in\partial\Om$
	with $\theta_*(\{u> b\},y)\ge \beta$, we have $y\in\{T_{\beta}u\ge  b\}$.
	Thus
	\[
	\mathcal H(B(x,\eps)\cap\partial\Om\cap \{\theta_*\{u>b\}\ge \beta\})<\infty.
	\]
	We find $j,k\in\N$ such that $B_j\subset B(x,\eps)$ and
	\[
	b=\min\{T_{*}u(x)-\eps,1/\eps\}< q_k<T_{*}u(x).
	\]
	Then
	\[
	\mathcal H(B_j\cap \partial \Om\cap \Sigma_{\beta}\{u>q_k\})
	\le \mathcal H(B_j\cap \partial \Om\cap \{\theta_*\{u>q_k\}\ge \beta\})<\infty.
	\]
	By \eqref{eq:def of theta}, also
	\[
	\mathcal H(B_j\cap \Om\cap \Sigma_{\beta}\{u>q_k\})<\infty,
	\]
	and clearly $B_j\cap \Sigma_{\beta}\{u>q_k\}\setminus \overline{\Om}=\emptyset$.
	Then $P(\{u>q_k\},B_j)<\infty$ by Theorem \ref{thm:new Federer}.
	Since $x\notin \Om_{\beta}$, necessarily $\theta_* \{u>q_k\}<\beta$, and since $x\notin N$, in fact
	$\theta^* \{u>q_k\}=0$. This implies that $T_*u(x)\le q_k$.
	This contradicts the fact that $q_k<T_*u(x)$.
	Thus \eqref{eq:finite measure contradiction} is false, proving the claim.
\end{proof}

For nice domains we get the following version.

\begin{proposition}\label{prop:weak continuity in nice domain}
	Suppose $\Om$ is PLB and let $u\in\BV(\Om)$.
	Then for $\mathcal H$-a.e. $x\in\partial \Om\setminus \Om_{\beta}$,  we have that if
	$T_*u(x)>-\infty$ and $\eps>0$, then
	\[
	\mathcal H(B(x,\eps )\cap \Om_{\beta}\cap \{Tu\ge \min\{T_{*}u(x)-\eps,1/\eps\}\})=\infty.
	\]
\end{proposition}

\begin{proof}
	By Proposition \ref{prop:weak continuity},
	for $\mathcal H$-a.e. $x\in\partial \Om\setminus \Om_{\beta}$  we have that if
	$T_*u(x)>-\infty$ and $\eps>0$, then
	\[
	\mathcal H(B(x,\eps )\cap \{T_{\beta}u\ge \min\{T_{*}u(x)-\eps,1/\eps\}\})=\infty.
	\]
	Note that necessarily $\{T_{\beta}u\ge \min\{T_{*}u(x)-\eps,1/\eps\}\}\subset \Om_{\beta}$.
By Proposition \ref{prop:existence of trace}, in the set $\partial\Om\cap \Om_{\beta}$ we have
that $Tu$ exists and thus
$T_{\beta}u=Tu$ $\mathcal H$-almost everywhere in this set, completing the proof.
\end{proof}

\begin{remark}
Clearly, the continuity property of Propositions \ref{prop:weak continuity}	and
\ref{prop:weak continuity in nice domain} does not generally hold at points $x\in\partial \Om\cap \Om_{\beta}$.
A counterexample is given simply by any domain with $\mathcal H(\partial\Om)<\infty$,
for example any bounded Lipschitz domain $\Om\subset \R^n$.
\end{remark}

\begin{lemma}\label{lem:trace absolute value comparison}
Let $u$ be a $\mu$-measurable function on $\Om$. Then
\[
|T_*|u|(x)|\ge |T_*u(x)|\quad\textrm{for every }x\in\partial\Om.
\]
\end{lemma}
\begin{proof}
	We can assume that $\theta^{*}(\Om,x)>0$.
If $T_*u(x)\ge 0$, the claim is obvious. If $T_*u(x):=b\in (-\infty,0)$, let $\delta>0$ and note that
\[
\theta^{*}(\{u> b+\delta/2\},x)=0
\]
and so
\[
\theta^*(\{u<  b+\delta\},x)
\ge \theta^*(\{u\le   b+\delta/2\},x)
= \theta^{*}(\Om,x)>0.
\]
Thus
\[
\theta^*(\{|u|>|b|-\delta\},x)
\ge \theta^*(\{-u>|b|-\delta\},x)
= \theta^*(\{u<b+\delta\},x)>0.
\]
Thus $T_*|u|(x)\ge |b|-\delta$, and since $\delta>0$ was arbitrary,
\[
T_*|u|(x)\ge |b|=|T_*u(x)|.
\]
The case $b=-\infty$ is similar.
\end{proof}

Now we can prove the following result on the integrability of the trace; this gives Claim (2) of
Theorem \ref{thm:main theorem}.

\begin{proposition}\label{prop:integrability}
	Suppose $\Om$ is PLB and let $u\in\BV(\Om)$.
	Then
	\[
	\int_{\partial\Om\cap \Om_{\beta}}|Tu|\,d\mathcal H
	=\int_{\partial\Om\cap \{\theta^*(\Om,\cdot)>0\}}|T_{*}|u||\,d\mathcal H
	=\int_{\partial\Om\cap \{\theta^*(\Om,\cdot)>0\}}|T_{*}u|\,d\mathcal H
	\]
\end{proposition}

\begin{proof}
	First note that the trace $Tu(x)$ exists at $\mathcal H$-a.e. $x\in \partial\Om\cap \Om_{\beta}$
	by Proposition \ref{prop:existence of trace}, and so the
	integral on the left-hand side is well-defined.
	It is also easy to check that whenever $Tu(x)$ exists, we have $|Tu(x)|=T|u|(x)$.
	Thus to prove the first equality, we can assume that $u\ge 0$.

	Let $N$ be the exceptional set of Proposition \ref{prop:weak continuity in nice domain}.
	Consider a point
	\[
	x\in\partial\Om\cap \{\theta^*(\Om,\cdot)>0\}\setminus (\Om_{\beta}\cup N).
	\]
	If $T_*u(x)\in (0,\infty]$, then Proposition \ref{prop:weak continuity in nice domain} with the choice $\eps=\min\{T_*u(x)/2,1\}$ gives
	\[
	\mathcal H(B(x,\eps )\cap \Om_{\beta}\cap \{Tu\ge  \eps\})=\infty.
	\]
	This means that
	\[
	\int_{\partial\Om\cap \Om_{\beta}}Tu\,d\mathcal H=\infty,
	\]
	and since $T_*u= Tu$ $\mathcal H$-almost everywhere in $\partial\Om\cap \Om_{\beta}$,
	then also
	\[
	\int_{\partial\Om\cap \{\theta^*(\Om,\cdot)>0\}}T_{*}u\,d\mathcal H=\infty,
	\]
	and so the first equality holds.
	The other option is that for every
	$x\in\partial\Om\cap \{\theta^*(\Om,\cdot)>0\}\setminus (\Om_{\beta}\cup N)$, we have
	$T_{*}u(x)=0$. Then
	\begin{align*}
	\int_{\partial\Om\cap \{\theta^*(\Om,\cdot)>0\}}T_{*}u\,d\mathcal H
	&=\int_{\partial\Om\cap (\Om_{\beta}\cup N)}T_{*}u\,d\mathcal H\\
	&=\int_{\partial\Om\cap \Om_{\beta}}T_{*}u\,d\mathcal H\quad\textrm{since }\mathcal H(N)=0\\
	&=\int_{\partial\Om\cap \Om_{\beta}}Tu\,d\mathcal H\quad\textrm{by Proposition }\ref{prop:existence of trace}.
	\end{align*}
	This completes the proof of the first equality.
	
	To prove the second equality, note that
	by Lemma \ref{lem:trace absolute value comparison} and the first equality, we have
	\[
	\int_{\partial\Om\cap \Om_{\beta}}|Tu|\,d\mathcal H\ge \int_{\partial\Om\cap \{\theta^*(\Om,\cdot)>0\}}|T_{*}u|\,d\mathcal H.
	\]
	The opposite inequality is obvious, since at $\mathcal H$-a.e.
	$x\in \partial\Om\cap \Om_{\beta}$ we have
	$Tu(x)=T_{*}u(x)$.
\end{proof}

In the setting of Example \ref{ex:measure theoretic and strong boundary},
one can for example consider various nonnegative functions $u\in \BV(\Om)$,
and find that always
\[
\int_{\partial\Om}T_{*}u\,d\mathcal H
=\int_{\partial\Om\setminus \{x_0\}}Tu\,d\mathcal H.
\]
The point is that if $T_*u(\{x_0\})>0$, such as when $u=1$ in $\Om$, then
the right-hand side takes the value $+\infty$ and so the two sides are equal.

Note that here we have been considering the part of the boundary $\partial\Om\cap \Om_{\beta}$, but in
Claim (3) of Theorem \ref{thm:main theorem} we consider a smaller set, namely the strong boundary
$\Sigma_{\beta}\Om=\Om_{\beta}\cap (X\setminus \Om)_{\beta}$. This raises the question of whether it would be possible
to formulate also Claim (2) in terms of the strong boundary. However, this is not possible.

\begin{example}\label{ex:strong boundary}
Let $X=\R^2$ equipped with the weighted Lebesgue measure
\[
d\mu:=w\,d\mathcal L^2,\quad\textrm{with}\quad w(x):=\frac{1}{2\pi}|x|^{-1}.
\]
Just as in Example \ref{ex:measure theoretic and strong boundary}, we have that
$(X,d_{\textrm{euc}},\mu)$ satisfies doubling and $(1,1)$-Poincar\'e.
Denote the origin by $0$.
Let $x_j:=(2^{-j},0)$, $r_j:=2^{-2j}$, and
\[
\Om:=B(0,1)\setminus \Bigg[\{0\}\cup \bigcup_{j=3}^{\infty}\overline{B(x_j,r_j)}\Bigg].
\]
It is straightforward to check that $\Om$ is a uniform domain;
recall the definition from \eqref{eq:uniformity}.
Thus by \cite[Theorem 4.4]{BS}, $(\Om,d_{\textrm{euc}},\mu)$ satisfies doubling and $(1,1)$-Poincar\'e.

Now clearly $\partial\Om=\partial\Om\cap \Om_{\beta}$ and $0\in \partial\Om$,
but $\theta^{*}(X\setminus \Om,0)=0$
so in particular $0\notin \Sigma_{\beta}\Om$.
Let $u:=1$ on $\Om$, so that $u\in\BV(\Om)$ and $Tu=T_*u=1$ on $\partial\Om$.
However, by \eqref{eq:measure of origin}
we have $\mathcal H(\{0\})>0$. Denoting the one-dimensional Hausdorff measure
by $\mathcal H^1$, we have
\begin{align*}
\int_{\partial\Om\cap \Om_{\beta}}|Tu|\,d\mathcal H
&=\mathcal H(\{0\})+\sum_{j=3}^{\infty}\mathcal H(\partial B_j)\\
&\le \mathcal H(\{0\})+\sum_{j=3}^{\infty}w((2^{-j-1},0))\cdot \pi\cdot \mathcal H^{1}(\partial B_j)\\
&\le \mathcal H(\{0\})+\sum_{j=3}^{\infty}\frac{1}{2\pi}\cdot 2^{j+1}\cdot \pi\cdot 2\pi \cdot 2^{-2j}<\infty.
\end{align*}
Now
\[
\int_{\Sigma_{\beta}\Om}|Tu|\,d\mathcal H
=\sum_{j=3}^{\infty}\mathcal H(\partial B_j)
<\int_{\partial\Om\cap \Om_{\beta}}|Tu|\,d\mathcal H.
\]
Thus in Claim (2) of Theorem \ref{thm:main theorem},
we cannot replace $\partial\Om\cap \Om_{\beta}$ with $\Sigma_{\beta}\Om$.
\end{example}

\section{Zero extension}\label{sec:zero extension}

In this section we give a characterization of those functions $u\in\BV(\Om)$
that can be zero extended to the whole space as BV functions.
We also study the possibility of doing this without adding any total variation.

As before, $\Om\subset X$ always denotes an arbitrary open set.

\begin{lemma}\label{lem:equality of Hausdorff measures}
Let $u$ be a measurable function on $\Om$. Let $A\subset \partial\Om$
be a Borel set. Then
\[
\mathcal H(\{\theta_*(\{u>t\},\cdot )\ge \beta\}\cap A)
=\mathcal H(\{T_{\beta}u>t\}\cap A)
\quad\textrm{for a.e. }t\in\R.
\]
\end{lemma}

\begin{proof}
	For every $t\in\R$ we have
	\begin{equation}\label{eq:s and t}
	\{T_{\beta}u> t\}
	\subset \{\theta_*(\{u>t\},\cdot)\ge \beta\}\cap \partial\Om,
	\end{equation}
	because if $x\in \partial\Om$ with
	$\theta_*(\{u>t\},x)<\beta$, then $T_{\beta}u(x)\le t$.
	
Choose $a\in [-\infty,\infty]$ such that $\mathcal H(\{T_{\beta}u>t\}\cap A)<\infty$ for all $t>a$ and 
$\mathcal H(\{T_{\beta}u>t\}\cap A)=\infty$ for all $t<a$.
There are at
most countably many $t>a$ for which
$\mathcal H(\{T_{\beta}u=t\}\cap A)>0$. Consider a number $t>a$ outside this countable set.
If $x\in A$ such that $T_{\beta}u(x)\neq t$ and $\theta_*(\{u>t\},x)\ge\beta$, then $T_{\beta}u(x)>t$.
Conversely if $T_{\beta}u(x)>t$, then $\theta_*(\{u>t\},x)\ge \beta$ by \eqref{eq:s and t}.
Thus
\[
\mathcal H(\{\theta_*(\{u>t\},\cdot)\ge \beta\}\cap A)
=\mathcal H(\{T_{\beta}u>t\}\cap A)
\quad\textrm{for a.e. }t>a.
\]

Then consider $t<a$. Now $\mathcal H(\{T_{\beta}u>t\}\cap A)=\infty$.
By \eqref{eq:s and t} we get
\[
\{T_{\beta}u> t\}\cap A
\subset
\{\theta_*(\{u>t\},\cdot)\ge \beta\}\cap A,
\]
and so also
\[
\mathcal H(\{\theta_*(\{u>t\},\cdot)\ge \beta\}\cap A)=\infty.
\]
Thus, the claim holds also for almost every (in fact, every) $t<a$.
\end{proof}

Now we prove the following characterization, which gives Claim (3) of Theorem \ref{thm:main theorem}.

\begin{proposition}\label{prop:extension}
Suppose $\Om$ is PLB and
let $u\in\BV(\Om)$. Then the zero extension of $u$ from $\Om$ to the whole space $X$ belongs to $\BV(X)$
if and only if
\[
\int_{\Sigma_{\beta}\Om}|Tu|\,d\mathcal H<\infty.
\]
\end{proposition}

\begin{proof}
For numbers $a,b\ge 0$, we write $a\approx b$ if $C^{-1} a\le b\le Ca$ for some constant $C\ge 1$
depending only on the doubling constant of $\mu$ and the constants in the Poincar\'e inequality
(that is, the doubling and Poincar\'e that hold globally in the space $X$).
Denote by $u$ also the zero extension; obviously $u\in L^1(X)$.
First assume that $u\ge 0$. Now by the coarea formula \eqref{eq:coarea}, we get the following;
note that we use upper integrals  since measurability is not clear,
and this is also the reason for the second ``$\approx$'':
\begin{equation}\label{eq:Du X calculation start}
\begin{split}
\Vert Du\Vert(X)
&=\int_{(-\infty,\infty)}^*P(\{u>t\},X)\,dt\\
&=\int_{(0,\infty)}^* P(\{u>t\},X)\,dt\\
&\approx \int_{(0,\infty)}^*\mathcal H(\Sigma_{\beta}\{u>t\})\,dt\quad\textrm{by }\eqref{eq:perimeter and strong boundary}\\
&\approx\int_{(0,\infty)}^*\mathcal H(\Sigma_{\beta}\{u>t\}\cap \Om)\,dt
+\int_{(0,\infty)}^*\mathcal H(\Sigma_{\beta}\{u>t\}\cap \partial\Om)\,dt\\
&\approx\int_{(0,\infty)}^*P(\{u>t\}, \Om)\,dt
+\int_{(0,\infty)}^*\mathcal H(\Sigma_{\beta}\{u>t\}\cap \partial\Om)\,dt\quad\textrm{by }\eqref{eq:perimeter and strong boundary}\\
&=\Vert Du\Vert(\Om)
+\int_{(0,\infty)}^*\mathcal H(\Sigma_{\beta}\{u>t\}\cap \partial\Om\cap \Om_{\beta})\,dt;
\end{split}
\end{equation}
the last inequality follows from the coarea formula \eqref{eq:coarea} and since obviously $\Sigma_{\beta}\{u>t\}\cap\partial\Om\subset \Om_{\beta}$.
Now let $V$ and $W$ be two sets as in Definition \ref{def:local conditions on boundary}.
By \cite[Proposition 3.3 \& 3.6]{BB-noncomp}, we have that
$(\overline{W},d,\mu_{\overline{W}})$ satisfies doubling and $(1,1)$-Poincar\'e.
Denote by $\overline{u}$ the zero extension of $u$ from $W$ to $\overline{W}$.
By the coarea formula \eqref{eq:coarea} and \eqref{eq:u BV in W closure}, we have
\begin{align*}
\int_{(0,\infty)}^*\mathcal H_{\overline{W}}(\partial^*\{\overline{u}>t\}\cap V \cap \partial\Om)\,dt
\approx \Vert D\overline{u}\Vert(V\cap \partial\Om)=0.
\end{align*}
Thus for a.e. $t>0$ we have
\[
\mathcal H_{\overline{W}}(\partial^*\{\overline{u}>t\}\cap V \cap \partial\Om)=0.
\]
Fix such $t$.
For $\mathcal H_{\overline{W}}$-a.e. $x\in V \cap \partial\Om$,
we have $x\notin\partial^*\{\overline{u}>t\}$ and so either
\[
\lim_{r\to 0}\frac{\mu(B(x,r)\cap W\cap \{u>t\})}{\mu(B(x,r)\cap W)}=0\quad\textrm{or}\quad
\lim_{r\to 0}\frac{\mu(B(x,r)\cap W\cap \{u\le t\})}{\mu(B(x,r)\cap W)}=0.
\]
For small enough $r>0$ we have $B(x,r)\cap \Om\subset V\cap\Om\subset W$.
Thus
\[
\lim_{r\to 0}\frac{\mu(B(x,r)\cap \Om\cap \{u>t\})}{\mu(B(x,r))}=0\quad\textrm{or}\quad
\lim_{r\to 0}\frac{\mu(B(x,r)\cap \Om\cap \{u\le t\})}{\mu(B(x,r))}=0.
\]
Now if $x\notin (X\setminus\Om)_{\beta}$, then either
\[
\liminf_{r\to 0}\frac{\mu(B(x,r)\cap \{u>t\})}{\mu(B(x,r))}<\beta\quad\textrm{or}\quad
\liminf_{r\to 0}\frac{\mu(B(x,r)\cap \{u\le t\})}{\mu(B(x,r))}<\beta,
\]
and so $x\notin \Sigma_{\beta}\{u>t\}$.
Thus for a.e. $t>0$ we have
\[
\mathcal H_{\overline{W}}(\Sigma_{\beta}\{u>t\}\cap V \cap \partial\Om\setminus (X\setminus\Om)_{\beta})=0.
\]
By Lemma \ref{lem:Hausdorff zero measure sets agree} and Lemma \ref{lem:H and overlineH}, we have in fact
\[
\mathcal H(\Sigma_{\beta}\{u>t\}\cap V \cap \partial\Om\cap\Om_{\beta}\setminus (X\setminus\Om)_{\beta})=0.
\]
As noted in the first paragraph of the proof of
Proposition \ref{prop:existence of trace}, we can cover $\mathcal H$-almost all of
$\partial\Om\cap \Om_{\beta}$ by countably many such sets $V$.
It follows that 
\[
\mathcal H(\Sigma_{\beta}\{u>t\}\cap \partial\Om\cap\Om_{\beta}\setminus (X\setminus\Om)_{\beta})=0
\quad\textrm{for a.e. }t>0.
\]
From \eqref{eq:Du X calculation start} we now get
\[
	\Vert Du\Vert(X)
	\approx \Vert Du\Vert(\Om)
	+\int_{(0,\infty)}^*\mathcal H(\Sigma_{\beta}\{u>t\}\cap \Sigma_{\beta}\Om)\,dt.
\]
Here
\begin{align*}
\int_{(0,\infty)}^*\mathcal H(\Sigma_{\beta}\{u>t\}\cap \Sigma_{\beta}\Om)\,dt
&=\int_{0}^{\infty}\mathcal H(\{\theta_*(\{u>t\},\cdot )\ge \beta\}\cap \Sigma_{\beta}\Om)\,dt\\
&=\int_{0}^{\infty}\mathcal H(\{T_{\beta}u>t\}\cap\Sigma_{\beta}\Om)\,dt\quad\textrm{by Lemma }\ref{lem:equality of Hausdorff measures}\\
&=\int_{\Sigma_{\beta}\Om}T_{\beta}u\,d\mathcal H\quad\textrm{by Cavalieri's principle}\\
&=\int_{\Sigma_{\beta}\Om}Tu\,d\mathcal H\quad\textrm{by Proposition }\ref{prop:existence of trace}.
\end{align*}
In total,
\begin{equation}\label{eq:Du X and Omega for u positive}
\Vert Du\Vert(X) \approx \Vert Du\Vert(\Om)+\int_{\Sigma_{\beta}\Om}Tu\,d\mathcal H.
\end{equation}

Now we drop the assumption $u\ge 0$.
For a general $u\in\BV(\Om)$, note first that whenever
the trace $Tu(x)$ exists at a point $x\in\partial\Om$, then we have
\[
\textrm{either }\ Tu(x)=Tu_+(x)\textrm{ and }Tu_-(x)=0, \ \textrm{ or }\ Tu(x)=-Tu_-(x)\textrm{ and }Tu_+(x)=0.
\]
In both cases we get
\begin{equation}\label{eq:Tu positive and negative part}
|Tu(x)|=Tu_+(x)+Tu_-(x).
\end{equation}
Using the coarea formula \eqref{eq:coarea}, it is easy to check that
$\Vert Du\Vert(X)=\Vert Du_+\Vert(X)+\Vert Du_-\Vert(X)$, and similarly with $X$ replaced by $\Om$.
Now we have
\begin{align*}
	\Vert Du\Vert(X)
	&=\Vert Du_+\Vert(X)+\Vert Du_-\Vert(X)\\
	&\approx\Vert Du_+\Vert(\Om)+\int_{\Sigma_{\beta}\Om}Tu_+\,d\mathcal H
	+\Vert Du_-\Vert(\Om)+\int_{\Sigma_{\beta}\Om}Tu_-\,d\mathcal H\quad\textrm{by }\eqref{eq:Du X and Omega for u positive}\\
	&=\Vert Du\Vert(\Om)+\int_{\Sigma_{\beta}\Om}|Tu|\,d\mathcal H\quad\textrm{by }\eqref{eq:Tu positive and negative part}.
\end{align*}
This proves the result.
\end{proof}

Now we have proved all the claims of our main Theorem \ref{thm:main theorem}:

\begin{proof}[Proof of Theorem \ref{thm:main theorem}]
Claim (1) is given by Proposition \ref{prop:existence of trace}, Claim (2) by
Proposition \ref{prop:integrability},
and Claim (3) by Proposition \ref{prop:extension}.
\end{proof}

\begin{remark}	
	In essence, the strategy of our proofs was the following:
	we divided the boundary $\partial\Om$
	into two parts, $\partial\Om\cap \Om_{\beta}$ and $\partial\Om\setminus \Om_{\beta}$.
	In the first part, we were able to show the existence of the (ordinary) trace.
	In the second part, we could control the relevant quantities by using the new version of Federer's characterization,
	Theorem \ref{thm:new Federer}.
\end{remark}

Consider a function $u\in N_0^{1,p}(\Om)$; recall the definition from \eqref{eq:definition of N011}.
Interpreting $u$ to be defined on the whole space,
we have $g_u=0$ $\mu$-almost everywhere in $X\setminus \Om$,
and $\Vert u\Vert_{N^{1,p}(\Om)}=\Vert u\Vert_{N^{1,p}(X)}$;
see \cite[Corollary 2.21 \& Proposition 2.38]{BB}.
In this sense, for Newton-Sobolev functions with zero boundary values,
zero extension to the whole space adds no energy.
In this section we have considered the possibility of zero extending BV functions to the whole space, but possibly
adding total variation on the boundary $\partial\Om$.
For example, if $u=\ch_{\Om}=\ch_{B(0,1)}$ in $\R^n$, zero extension adds
total variation on the boundary $\partial\Om$.
Now we study the possibility of extending without adding any total variation.
Recall the definition of $u^{\vee}$ from \eqref{eq:upper approximate limit}.

\begin{theorem}\label{thm:BV zero class characterization}
Let $\Om\subset \Om_0\subset X$ be open sets.
Let $u\in\BV(\Om)$. Define the zero extension
\[
u_0:=
\begin{cases}
u &\textrm{in }\Om,\\
0 &\textrm{in }\Om_0\setminus \Om.
\end{cases}
\] Then
the following are equivalent:
\begin{enumerate}[(1)]
	\item $u_0\in\BV(\Om_0)$ with $\Vert Du_0\Vert(\Om_0\setminus \Om)=0$ and $u_0^{\vee}(x)=0$
	for $\mathcal H$-a.e. $x\in\Om_0\cap \partial\Om$.
	\item For $\mathcal H$-a.e. $x\in \Om_0\cap \partial \Om$, we have
	\[
	\lim_{r\to 0}\frac{1}{\mu(B(x,r))}\int_{B(x,r)\cap\Om}|u|\,d\mu=0.
	\]
	\item For $\mathcal H$-a.e. $x\in \Om_0\cap \partial \Om$, we have
	\begin{equation}\label{eq:Lebesgue point type condition}
	\liminf_{r\to 0}\frac{1}{\mu(B(x,r))}\int_{B(x,r)\cap\Om}|u|\,d\mu=0.
	\end{equation}
\end{enumerate}
\end{theorem}

\begin{proof}
	$(1)\Rightarrow (2)$: By the decomposition \eqref{eq:variation measure decomposition} we know that
	$\mathcal H(S_{u_0}\cap\Om_0\setminus \Om)=0$.
	Then the Lebesgue point result \eqref{eq:Lebesgue points} gives for
	$\mathcal H$-a.e. $x\in \Om_0\setminus\Om$, and in particular for $\mathcal H$-a.e. $x\in \Om_0\cap \partial \Om$, that
	\[
	\lim_{r\to 0}\frac{1}{\mu(B(x,r))}\int_{B(x,r)\cap\Om}|u|\,d\mu
	=\lim_{r\to 0}\frac{1}{\mu(B(x,r))}\int_{B(x,r)}|u_0-u_0^{\vee}(x)|\,d\mu=0.
	\]
	
	$(2)\Rightarrow (3)$: This is obvious.
	
	$(3)\Rightarrow (1)$: Fix $x\in\Om_0\cap \partial\Om$ such 
	that \eqref{eq:Lebesgue point type condition} holds.
	If $t<0$, then
	\begin{equation}\label{eq:t larger than zero density}
	\begin{split}
		\liminf_{r\to 0}\frac{\mu(B(x,r)\setminus \{u_0>t\})}{\mu(B(x,r))}
		&=\liminf_{r\to 0}\frac{\mu(B(x,r)\cap \Om\setminus \{u>t\})}{\mu(B(x,r))}\\
		&\le \liminf_{r\to 0} \frac{1}{|t|\mu(B(x,r))}\int_{B(x,r)\cap \Om}|u|\,d\mu=0.
	\end{split}
	\end{equation}
	If $t>0$, then
	\begin{equation}\label{eq:t smaller than zero density}
	\begin{split}
		\liminf_{r\to 0}\frac{\mu(B(x,r)\cap \{u_0>t\})}{\mu(B(x,r))}
		&=\liminf_{r\to 0}\frac{\mu(B(x,r)\cap \Om\cap \{u>t\})}{\mu(B(x,r))}\\
		&\le \liminf_{r\to 0} \frac{1}{|t|\mu(B(x,r))}\int_{B(x,r)\cap \Om}|u|\,d\mu=0.
	\end{split}
	\end{equation}
	In both cases it follows that $x\notin \Sigma_{\beta}\{u_0>t\}$.
	Clearly this is true also for
	every $x\in \Om_0\setminus \overline{\Om}$.
	In conclusion, $\mathcal H(\Sigma_{\beta}\{u_0>t\}\cap \Om_0\setminus \Om)=0$
	for every $t\neq 0$.
	By the coarea formula \eqref{eq:coarea} we know that
	\[
	P(\{u_0>t\},\Om)=P(\{u>t\},\Om)<\infty\quad\textrm{for a.e. }t\in\R. 
	\]
	For such $t\neq 0$, by \eqref{eq:def of theta} we now have
	\[
	\mathcal H(\Sigma_{\beta}\{u_0>t\}\cap \Om_0)
	=\mathcal H(\Sigma_{\beta}\{u_0>t\}\cap \Om)
	\le \alpha^{-1} P(\{u_0>t\},\Om)<\infty.
	\]
	By Theorem \ref{thm:new Federer} it follows that $P(\{u_0>t\},\Om_0)<\infty$.
	Since $\mathcal H(\partial^*\{u_0>t\}\cap \Om_0\setminus \Om)=0$,
	by \eqref{eq:def of theta} we have $P(\{u_0>t\},\Om_0\setminus \Om)=0$.
	Again by the coarea formula,
	\[
	\Vert Du_0\Vert(\Om_0)=\int_{-\infty}^\infty P(\{u_0>t\},\Om_0)\, dt
	=\int_{-\infty}^\infty P(\{u_0>t\},\Om)\, dt
	=\Vert Du_0\Vert(\Om).
	\]
	It follows that $u_0\in \BV(\Om_0)$ with $\Vert Du_0\Vert(\Om_0\setminus \Om)=0$.
	From \eqref{eq:t larger than zero density} and \eqref{eq:t smaller than zero density}
	we also get $u_0^{\vee}(x)=0$
	for $\mathcal H$-a.e. $x\in\Om_0\cap \partial\Om$.
\end{proof}

Conditions (1) and (2) of the above theorem were shown to be equivalent already in \cite[Theorem 6.1]{LaSh}
as well as in \cite[Theorem 4.5]{L-Leib}.
By exploiting the new version of Federer's characterization
(Theorem \ref{thm:new Federer}), we have been able to prove that (3) is equivalent as well. 
The proofs given in the above references as well as in Theorem \ref{thm:BV zero class characterization}
follow along the lines of \cite{KKST12},
where a characterization of Newton-Sobolev functions with zero boundary values was proved.
We close by giving the following new version of this characterization.
Recall the definition of the $p$-capacity from \eqref{eq:definition of capacity}.

\begin{theorem}
	Suppose $\Om\subset X$ is open and bounded.
	Let $1\le p<\infty$ and let $u\in N^{1,p}(\Om)$.
	Then the following are equivalent:
	\begin{enumerate}
		\item $u\in N^{1,p}_0(\Om)$.
		\item For $\capa_p$-almost every $x\in \partial \Om$, we have
		\[
		\lim_{r\to 0}\frac{1}{\mu(B(x,r))}\int_{B(x,r)\cap\Om}|u|\,d\mu=0.
		\]
		\item For $\capa_p$-almost every $x\in \partial \Om$, we have
		\begin{equation}\label{eq:zero boundary liminf condition}
		\liminf_{r\to 0}\frac{1}{\mu(B(x,r))}\int_{B(x,r)\cap\Om}|u|\,d\mu=0.
		\end{equation}
	\end{enumerate}
\end{theorem}

\begin{proof}
	$(1)\Longleftrightarrow (2)$:
	This is shown in \cite[Theorem 1.1]{KKST12};
	note that the ``strong relative isoperimetric inequality''
	mentioned in the statement of that theorem is proved in \cite[Corollary 5.6]{L-Fedchar}.
	
	$(2)\Rightarrow (3)$:
	This is trivial.
	
	$(3)\Rightarrow (1)$:
	We have $u\in N^{1,p}(\Om)\subset N^{1,1}(\Om)$, since $\Om$ is bounded.
	From the definition of the total variation \eqref{eq:total variation} it follows that
	$N^{1,1}(\Om)\subset \BV(\Om)$, and so we have $u\in\BV(\Om)$.
	By \cite[Proposition 2.46]{BB}, the condition \eqref{eq:zero boundary liminf condition} holds also for
	$\capa_1$-almost every $x\in\partial\Om$.
	Then by \eqref{eq:null sets of Hausdorff measure and capacity},
	it holds for $\mathcal H$-a.e. $x\in\partial\Om$.
	Now by Theorem \ref{thm:BV zero class characterization}, we get
	$u_0\in\BV(X)$ with $\Vert Du_0\Vert(X\setminus\Om)=0$. After this, we can follow
	almost verbatim the proof given in
	\cite[Theorem 1.1]{KKST12} (p. 521).
\end{proof}

\noindent Address:\\

\noindent Academy of Mathematics and Systems Science\\
Chinese Academy of Sciences\\
Beijing 100190, PR China\\
E-mail: {\tt panulahti@amss.ac.cn}


\begin{thebibliography}{ACMM}

\bibitem{A1}L. Ambrosio,
\textit{Fine properties of sets of finite perimeter in doubling metric measure spaces},
Calculus of variations, nonsmooth analysis and related topics.
Set-Valued Anal. 10 (2002), no. 2-3, 111--128.

\bibitem{AFP}L. Ambrosio, N. Fusco, and D. Pallara,
\textit{Functions of bounded variation and free discontinuity problems.}
Oxford Mathematical Monographs. The Clarendon Press, Oxford University Press, New York, 2000.
xviii+434 pp.

\bibitem{AMP}L. Ambrosio, M. Miranda, Jr., and D. Pallara,
\textit{Special functions of bounded variation in doubling metric measure spaces},
Calculus of variations: topics from the mathematical heritage of E. De Giorgi, 1--45,
Quad. Mat., 14, Dept. Math., Seconda Univ. Napoli, Caserta, 2004.

\bibitem{BB}A. Bj\"orn and J. Bj\"orn,
\textit{Nonlinear potential theory on metric spaces},
EMS Tracts in Mathematics, 17. European Mathematical Society (EMS), Z\"urich, 2011. xii+403 pp.

\bibitem{BB-noncomp}A. Bj\"orn and J. Bj\"orn,
\textit{Poincar\'e inequalities and Newtonian Sobolev functions on noncomplete metric spaces},
J. Differential Equations 266 (2019), no. 1, 44--69.

\bibitem{BS}J. Bj\"orn and N. Shanmugalingam,
\textit{Poincar\'e inequalities, uniform domains and extension properties for Newton-Sobolev functions in metric spaces},
J. Math. Anal. Appl. 332 (2007), no. 1, 190--208.

\bibitem{BuMi}V. Buffa and M. Miranda, Jr.,
\textit{Rough traces of BV functions in metric measure spaces},
preprint 2019.
https://arxiv.org/abs/1907.01673

\bibitem{EvGa}L. C. Evans and R. F. Gariepy,
\textit{Measure theory and fine properties of functions},
Revised edition. Textbooks in Mathematics. CRC Press, Boca Raton, FL, 2015. xiv+299 pp.

\bibitem{Fed}H. Federer,
\textit{Geometric measure theory},
Die Grundlehren der mathematischen Wissenschaften, Band 153 Springer-Verlag New York Inc., New York 1969 xiv+676 pp.

\bibitem{Hj}P. Haj\l{}asz,
\textit{Sobolev spaces on metric-measure spaces},
Heat kernels and analysis on manifolds, graphs, and metric spaces (Paris, 2002), 173--218,
Contemp. Math., 338, Amer. Math. Soc., Providence, RI, 2003.

\bibitem{HaKi}H. Hakkarainen and J. Kinnunen,
\textit{The BV-capacity in metric spaces},
Manuscripta Math. 132 (2010), no. 1-2, 51--73.

\bibitem{HKM}J. Heinonen, T. Kilpel\"ainen, and O. Martio,
\textit{Nonlinear potential theory of degenerate elliptic equations},
Unabridged republication of the 1993 original. Dover Publications, Inc., Mineola, NY, 2006. xii+404 pp.

\bibitem{HK}J. Heinonen and P. Koskela,
\textit{Quasiconformal maps in metric spaces with controlled geometry},
Acta Math. 181 (1998), no. 1, 1--61.

\bibitem{HKST}J. Heinonen, P. Koskela, N. Shanmugalingam, and J. Tyson,
\textit{Sobolev spaces on metric measure spaces.
An approach based on upper gradients},
New Mathematical Monographs, 27. Cambridge University Press, Cambridge, 2015. xii+434 pp.

\bibitem{KKST12}J. Kinnunen, R. Korte, N. Shanmugalingam, and H. Tuominen,
\textit{A characterization of Newtonian functions with zero boundary values},
Calc. Var. Partial Differential Equations 43 (2012), no. 3-4, 507--528.

\bibitem{KKST}J. Kinnunen, R. Korte, N. Shanmugalingam, and H. Tuominen,
\textit{Pointwise properties of functions of bounded variation in metric spaces},
Rev. Mat. Complut. 27 (2014), no. 1, 41--67.

\bibitem{L-newFed}P. Lahti,
\textit{A new Federer-type characterization of sets of finite perimeter in metric spaces},
Arch. Ration. Mech. Anal. 236 (2020), no. 2, 801--838.

\bibitem{L-Leib}P. Lahti,
\textit{A sharp Leibniz rule for BV functions in metric spaces},
Rev. Mat. Complut. 33 (2020), no. 3, 797--816.

\bibitem{L-Fedchar}P. Lahti,
\textit{Federer's characterization of sets of finite perimeter in metric spaces},
Anal. PDE 13 (2020), no. 5, 1501--1519.

\bibitem{LaSh}P. Lahti and N. Shanmugalingam,
\textit{Trace theorems for functions of bounded variation in metric spaces},
J. Funct. Anal. 274 (2018), no. 10, 2754--2791.

\bibitem{MSS}L. Mal\'y, N. Shanmugalingam, M. Snipes,
\textit{Trace and extension theorems for functions of bounded variation},
Ann. Sc. Norm. Super. Pisa Cl. Sci. (5) 18 (2018), no. 1, 313--341.

\bibitem{Maz}V. G. Maz'ya,
\textit{Sobolev spaces with applications to elliptic partial differential equations},
Second, revised and augmented edition. Grundlehren der Mathematischen Wissenschaften [Fundamental Principles of Mathematical Sciences], 342. Springer, Heidelberg, 2011. xxviii+866 pp.

\bibitem{M}M. Miranda, Jr.,
\textit{Functions of bounded variation on ``good'' metric spaces},
J. Math. Pures Appl. (9) 82  (2003),  no. 8, 975--1004.

\bibitem{S}N. Shanmugalingam,
\textit{Newtonian spaces: An extension of Sobolev spaces to metric measure spaces},
Rev. Mat. Iberoamericana 16(2) (2000), 243--279.

\end{thebibliography}
\end{document}